\documentclass{amsart}
\usepackage{amssymb}
\newcommand\SL{\operatorname{SL}}
\newcommand\GL{\operatorname{GL}}
\newcommand\Aut{\operatorname{Aut}}
\newcommand\Out{\operatorname{Out}}
\newcommand\Comm{\operatorname{Comm}}
\newcommand\MCG{\operatorname{MCG}}
\newcommand\Z{\mathbb Z}

\newcommand\id{\operatorname{id}}
\newcommand\Stab{\operatorname{Stab}}

\renewcommand\index[2]{|#1:#2|}
\newcommand\subindex[2]{|#1:\,:#2|}
\newcommand\restrict[1]{\!\downharpoonleft_{\!#1}}

\newtheorem{lem}{Lemma}[section]
\newtheorem{thm}[lem]{Theorem}

\newtheorem{prop}[lem]{Proposition}
\newtheorem{mainthm}{Theorem}
\newtheorem{maincor}[mainthm]{Corollary}

\begin{document}
\title{On abstract commensurators of groups} \author{L. Bartholdi}
\address[L.B.]{Mathematisches Institut\\
  Bunsenstra\ss e 3--5\\
  G.-A. Universit\"at zu G\"ottingen\\
  37073 G\"ottingen\\
  Germany} \email{laurent.bartholdi@gmail.com} \author{O. Bogopolski}
\address[O.B.]{Institute of Mathematics\\
  Siberian Branch of the Russian Academy of Sciences\\
  Novosibirsk\\
  Russia, \textup{and}}
\address{Mathematisches Institut der H.-H. Universit\"at D\"usseldorf\\
  Universit\"atsstrasse 1\\
  40225 D\"usseldorf\\
  Germany}
\email{Oleg\_Bogopolski@yahoo.com}
\date{February 25, 2009}
\begin{abstract}
  We prove that the abstract commensurator of a nonabelian free group,
  an infinite surface group, or more generally of a group that splits
  appropriately over a cyclic subgroup, is not finitely generated.

  This applies in particular to all torsion-free word-hyperbolic
  groups with infinite outer automorphism group and abelianization of
  rank at least 2.

  We also construct a finitely generated, torsion-free group which can
  be mapped onto $\Z$ and which has a finitely generated
  commensurator.
\end{abstract}
\maketitle

\section{Introduction}

Let $G$ be a group. Consider the set $\Omega (G)$ of all isomorphisms
between subgroups of finite index of $G$. Two such isomorphisms
$\varphi_1:H_1\to H_1'$ and $\varphi_2:H_2\to H_2'$ are called
\emph{equivalent}, written $\varphi_1\sim\varphi_2$, if there exists a
subgroup $H$ of finite index in $G$ such that both $\varphi_1$ and
$\varphi_2$ are defined on $H$ and $\varphi_1\restrict
H=\varphi_2\restrict H$.

For any two isomorphisms $\alpha:G_1\to G_1'$ and $\beta:G_2\to G_2'$
in $\Omega(G)$, we define their product
$\alpha\beta:\alpha^{-1}(G_1'\cap G_2)\to \beta(G_1'\cap G_2)$ in
$\Omega (G)$. The factor-set $\Omega(G)/{\sim}$ inherits the
multiplication $[\alpha][\beta]=[\alpha\beta]$ and is a group, called
the \emph{abstract commensurator} of $G$ and denoted $\Comm(G)$.

$\Comm(G)$ is in general much larger than $\Aut(G)$. For example
$\Aut(\Z^n) \cong \GL(n,\Z)$ whereas $\Comm(\Z^n)
\cong \GL(n,\mathbb Q)$. Margulis proved that an irreducible lattice
$\Lambda$ in a semisimple Lie group $G$ is arithmetic if and only if
it has infinite index in its \emph{relative commensurator in $G$},
$$\Comm_G(\Lambda):=\{g\in G : g\Lambda g^{-1} \cap \Lambda \hspace*{2mm}{\text{\rm has finite index in both}}\hspace*{2mm} \Lambda\hspace*{2mm} {\text{\rm and}}\hspace*{2mm} g\Lambda g^{-1}\}.$$
`Mostow-Prasad-Margulis strong rigidity' for irreducible lattices
$\Lambda$ in $G\neq \SL (2,\mathbb R)$ implies that
the abstract commensurator $\Comm(\Lambda)$ is isomorphic to the
commensurator of $\Lambda$ in $G$, which in turn is computed
concretely by Margulis and Borel-Harish-Chandra; see
e.g.~\cite{Ma,Zi}. Analogously, for many groups acting on rooted
trees, their abstract commensurator equals their relative
commensurator in the automorphism group of the tree~\cite{Ro}.

Few abstract commensurators were explicitly computed. The group
$\Comm(\MCG_g)$ was computed for surface mapping class groups $\MCG_g$
by Ivanov~\cite{Iv}.  Farb and Handel proved in~\cite{FH} that
$\Comm(\Out(F_n))\cong \Out(F_n)$ for $n\geq 4$.  Leininger and
Margalit~\cite{LM} computed the abstract commensurator of the braid
group $B_n$ on $n\geq 4$ strings: $\Comm(B_n)\cong
(\mathbb{Q}^{\infty}\rtimes \mathbb{Q}^{\ast})\rtimes \MCG_{0,n+1}$,
where $\MCG_{0,n+1}$ is the mapping class group of the sphere with
$n+1$ punctures.
%Note that the groups $\Out(\MCG_g)$,
%$\Out(\Out(F_n))$ and $\Out(B_n)$ are finite, see
%respectively~\cite{Iv}, \cite{Kh}, and~\cite{Dy}.

Clearly, if $G$ is finitely generated, then $\Comm(G)$ is countable.
We show that, in many cases, it may be `large' in the sense that it is
not finitely generated. The cases we consider are groups $G$ which
split into an amalgamated product or an HNN extension over $1$ or
$\Z$, and satisfy some technical assumptions (see
Theorems~\ref{lem:free},~\ref{thm:HNN} and~\ref{thm:freeprod}).
%Note that $\Out(G)$
%is infinite for these groups (Proposition~\ref{prop:out}).
We deduce for example

\begin{maincor}\label{cor:surf}
  Let $G$ be either a non-Abelian free group, or a surface group
  $\pi_1(S)$ where $S$ is a closed surface of negative Euler
  characteristic. Then $\Comm(G)$ is not finitely generated.
\end{maincor}

\noindent Using a result by Paulin~\cite{Pau}, we deduce the more general
\begin{maincor}\label{cor:2}
  Let $G$ be a torsion-free word-hyperbolic group with infinite
  $\Out(G)$; suppose that $G$ can be homomorphically mapped onto
  $\Z\times\Z$.  Then $\Comm(G)$ is not finitely generated.
\end{maincor}

We then consider some possible relaxations of the hypotheses (in
particular (2) in Theorem~\ref{thm:freeprod}), and show that in each
case there are $G$ with finitely generated commensurator:
\begin{mainthm}\label{cor:Z}
  (1) There exist word-hyperbolic groups with finitely generated
  commensurator.

  (2) There exists a finitely generated group which has the unique
  root property (so in particular is torsion-free), can be mapped onto
  $\Z$, and whose commensurator is finitely generated.
\end{mainthm}
It is a fundamental open question as whether all word-hyperbolic
groups are residually finite; this is actually equivalent to asking
whether all word-hyperbolic groups have arbitrarily large
quotients~\cite[Theorem~2]{Olsh}. Of course, a group with no
non-trivial finite quotient has identical automorphism group and
abstract commensurator.

We start, in the next section, by a sufficient condition to ensure
that an abstract commensurator cannot be finitely generated.

%%%%%%%%%%%%%%%%%%%%%%%%%%%%%%%%%%%%%%%%%%%%%%%%%%%%%%%%%%%%%%%%
\section{Infinitely generated abstract commensurators}
Two groups $G,H$ are \emph {abstractly commensurable} if there exist
finite index subgroups $G_1\leqslant G$ and $H_1\leqslant H$, such
that $G_1\cong H_1$. The following useful lemma is well-known; for
completeness we give its proof.

\begin{lem}\label{lem:comm}
  If $G$ and $H$ are abstractly commensurable groups, then
  $\Comm(G)\cong \Comm(H)$.
\end{lem}

\begin{proof}
  Without loss of generality we can assume that $H$ is a subgroup of
  finite index in $G$. The embedding of $H$ in $G$ induces a canonical
  map $\Psi: \Comm(H)\rightarrow \Comm(G)$.  Now we define a map
  $\Phi: \Comm(G)\rightarrow \Comm(H)$ by the rule: for
  $\alpha:G_1\rightarrow G_2$ from $\Comm(G)$ we set
  $\Phi(\alpha)=\alpha\restrict{H_1}:H_1\rightarrow H_2$, where
  $H_1=\alpha^{-1}(G_2\cap H)\cap H$ and $H_2=\alpha(G_1\cap H)\cap
  H$. Clearly $\Phi(\alpha)$ belongs to $\Comm(H)$.  We leave it to
  the reader to check that $\Psi$ and $\Phi$ are homomorphisms, and
  that both compositions $\Psi\circ \Phi$ and $\Phi\circ \Psi$ are the
  identity.
\end{proof}

A group $G$ has the \emph{unique root property} if for any $x,y\in G$
and any positive integer $n$, the equality $x^n=y^n$ implies
$x=y$. Groups with the unique root property are torsion-free. It is
well known that, in torsion-free word-hyperbolic groups, nontrivial
elements have cyclic centralizers~\cite[pages 462--463]{BH}; so they
have the unique root property, by the following standard
\begin{lem}\label{lem:urp}
  Let $G$ be a torsion-free group with cyclic centralizers of
  nontrivial elements. Then $G$ has the unique root property.
\end{lem}
\begin{proof}
  Let $x,y$ be nontrivial elements of $G$. If $x^n=y^n$, then
  $Z(x^n)\geqslant\langle x,y\rangle$. But $Z(x^n)=\langle z\rangle$
  for some $z$, so there are $p,q\in\Z$ with $x=z^p$ and $y=z^q$. Then
  $x^n=y^n$ gives $z^{pn}=z^{qn}$, so $p=q$ and $x=y$.
\end{proof}

The usefulness of the unique root property can be seen immediately in
the following two lemmas.

\begin{lem}\label{lem:aut}
  Let $G$ be a group with the unique root property. Then $\Aut(G)$
  naturally embeds in $\Comm(G)$.
\end{lem}
\begin{proof}
  There is a natural homomorphism $\Aut(G)\rightarrow
  \Comm(G)$. Suppose that some $\alpha\in \Aut(G)$ lies in its
  kernel. Then $\alpha\restrict H=\id$ for some subgroup $H$ of finite index
  in $G$. If $m$ is this index, then $g^{m!}\in H$ for every $g\in
  G$. Then $\alpha(g^{m!})=g^{m!}$. Extracting roots, we get
  $\alpha(g)=g$, that is $\alpha=\id$.
\end{proof}

\begin{lem}\label{lem:restrict}
  Let $G$ be a group with the unique root property. Let $\varphi_1:
  H_1\to H_1'$ and $\varphi_2: H_2\to H_2'$ be two isomorphisms
  between subgroups of finite index in $G$.  Suppose that
  $[\varphi_1]=[\varphi_2]$ in $\Comm(G)$. Then $\varphi_1\restrict{H_1\cap
    H_2}=\varphi_2\restrict{H_1\cap H_2}$.
\end{lem}
\begin{proof}
  The equality $[\varphi_1]=[\varphi_2]$ means that there exists a
  subgroup $H$ of finite index in $G$ such that both $\varphi_1$ and
  $\varphi_2$ are defined on $H$ and
  $\varphi_1\restrict H=\varphi_2\restrict H$. Clearly $H\leqslant H_1\cap
  H_2$. Denote $m=\index{(H_1\cap H_2)}H$.  Let $h$ be an arbitrary element
  of $H_1\cap H_2$. Then $h^{m!}\in H$ and so
  $\varphi_1(h^{m!})=\varphi_2(h^{m!})$. Since $G$ is a group with the
  unique root property, we get $\varphi_1(h)=\varphi_2(h)$.
\end{proof}

Let us call the \emph{subindex} of a finite-index subgroup $H\leqslant
G$ the minimal $n$, denoted $\subindex GH$, such that there exists a
sequence of subgroups $H=G_0\leqslant G_1\leqslant\cdots\leqslant
G_k=G$ with $\index{G_i}{G_{i-1}}\le n$ for all
$i\in\{1,\dots,k\}$. Observe that given $F\leqslant H\leqslant G$, we
have $\subindex GF\le \max \{\subindex GH, \subindex HF\}$.

\begin{lem}\label{lem:subindex}
  Let $G$ be a group and let $\alpha_i:H_i\to H_i'$, for $i=1,\dots,r$
  be isomorphisms between subgroups of finite index of $G$. Assume
  that $\subindex G{H_i}\le n$ and $\subindex G{H'_i}\le n$ for all
  $i$. Then any finite product of $[\alpha_i]$'s can be realized by an
  isomorphism $\beta: H\to H'$, where $H,H'$ are subgroups of finite
  index and subindex at most $n$.
\end{lem}

\begin{proof}
  By induction, it suffices to consider $\alpha_1:H_1\to H'_1$ and
  $\alpha_2:H_2\to H'_2$, and their product
  $\beta=\alpha_1\alpha_2$. Set $K=H'_1\cap H_2$,  $H=\alpha_1^{-1}(K)$ and $H'=\alpha_2(K)$, so that
  $\beta:H\to H'$. Let
  $H_2=G_0\leqslant G_1\leqslant \cdots\leqslant G_k=G$ be a sequence of subgroups with
  $\subindex{G_i}{G_{i-1}}\le n$. The sequence $K=G_0\cap H_1'\leqslant G_1\cap
  H_1'\leqslant\cdots\leqslant G_k\cap H_1'=H_1'$ shows that $\subindex{H'_1}K\le n$.
  Then
  $$\subindex GH\le\max\{\subindex
  G{H_1},\subindex{H_1}H\}=\max\{\subindex
  G{H_1},\subindex{H'_1}K\}\le n;$$
  and similarly $\subindex G{H'}\le n$.
\end{proof}

\begin{lem}\label{lem:extend}
  Let $G$ be a group with the unique root property. Let $\varphi_1:
  H_1\to H_1'$ and $\varphi_2: H_2\to H_2'$ be two isomorphisms
  between subgroups of finite index in $G$.  Suppose that
  \begin{enumerate}
  \item $H_2$ is a normal subgroup of $G$;
  \item $\varphi_1\restrict{H_1\cap H_2}=\varphi_2\restrict{H_1\cap
      H_2}$.
  \end{enumerate}
  
  Then $\varphi_1,\varphi_2$ have a common extension, that is there
  exists an isomorphism $\varphi:H_1H_2\to H_1'H_2'$, such that
  $\varphi\restrict{H_i}=\varphi_i$ for $i=1,2$.
\end{lem}
\begin{proof}
  We define $\varphi:H_1H_2\to H_1'H_2'$ by
  $\varphi(h_1h_2)=\varphi_1(h_1)\varphi_2(h_2)$ for any $h_1\in H_1$
  and $h_2\in H_2$. This definition is unambiguous because of Property
  (2). We prove first that $\varphi$ is a homomorphism.

  Take $x\in H_1H_2$ and $y\in H_1H_2$. Then $x=g_1g_2$ and $y=h_1h_2$ for
  some $g_1,h_1\in H_1$ and $g_2,h_2\in H_2$. Since $xy=g_1h_1\cdot
  h_1^{-1}g_2h_1h_2$, where $h_1^{-1}g_2h_1\in H_2$ by Property (1),
  we have
  $$\varphi(xy)=\varphi_1(g_1)\varphi_1(h_1)\cdot \varphi_2(h_1^{-1}g_2h_1)\varphi_2(h_2).$$
  On the other hand we have
  $$\varphi(x)\varphi(y)=\varphi_1(g_1)\varphi_2(g_2)\varphi_1(h_1)\varphi_2(h_2).$$
  Thus it is enough to verify
  that
  \begin{equation}\label{eq:*}
    \varphi_2(h_1^{-1}g_2h_1)=\varphi_1(h_1)^{-1}\varphi_2(g_2)\varphi_1(h_1).\tag{*}
  \end{equation}
  Since $H_1\cap H_2$ has finite index in $H_2$, we have $g_2^m\in
  H_1\cap H_2$ for some positive integer $m$. Then
  $h_1^{-1}g_2^mh_1\in H_1\cap H_2$ and so
  $$\varphi_2(h_1^{-1}g_2^mh_1)=\varphi_1(h_1^{-1}g_2^mh_1)=\varphi_1(h_1^{-1})\varphi_1(g_2^m)\varphi_1(h_1)=
  \varphi_1(h_1)^{-1}\varphi_2(g_2)^m\varphi_1(h_1).$$ Since $G$ is a
  group with the unique root property, we can extract $m$-th roots
  from both sides of the last equation and get~\eqref{eq:*}.

  Clearly $\varphi$ maps onto $H'_1H'_2$. Assume for contradiction that
  $\varphi$ is not injective; then, since $G$ is torsion-free,
  $\ker\varphi$ is infinite. Since $H_1$ has finite index,
  $\ker\varphi\cap H_1$ is non-trivial, so $\varphi_1$ is not
  injective, a contradiction.
\end{proof}

\begin{thm}\label{thm:infgen}
  Let $G$ be a group with the unique root property. Suppose that, for
  infinitely many primes $p$, there exists a normal subgroup $H$ of
  index $p$ in $G$ and an automorphism of $H$ that cannot be extended
  to an automorphism of $G$.

  Then the commensurator of $G$ is not finitely generated.
\end{thm}
\begin{proof}
  Suppose that $\Comm(G)$ is generated by a finite number of classes
  of isomorphisms $\alpha_i: H_i\to H_i'$, for $i=1,\dots,k$, where
  $H_i,H_i'$ are subgroups of finite index in $G$.  Set
  $n=\max\{\subindex G{H_i},\subindex G{H'_i}:\,i=1,\dots,k\}$.

  Now take a prime number $p>n$. By assumption, there exists a normal
  subgroup $H$ of index $p$ in $G$ and an automorphism $\beta$ of $H$,
  which cannot be extended to an automorphism of $G$.

  Clearly $[\beta]\in \Comm(G)$. By Lemma~\ref{lem:subindex}, the class
  $[\beta]$ can be realized by an isomorphism $\alpha:A\to B$, where
  $A,B$ are subgroups of finite index in $G$ and subindex at most $n$.
  By Lemma~\ref{lem:restrict}, the automorphisms $\beta$ and $\alpha$
  coincide on the subgroup $H\cap A$.

  By Lemma~\ref{lem:extend}, the automorphism $\beta$ can be extended
  to an isomorphism $\varphi:AH\to BH$. Note that $AH=BH=G$ because
  the indices of $A$ and $H$ are coprime and the indices of $B$ and
  $H$ are coprime. We have reached a contradiction.
\end{proof}

We shall also need a variant of the previous result:
\begin{prop}\label{prop:infgen}
  Let $G$ be a group with the unique root property. Suppose that, for
  infinitely many primes $p$, there exists a subgroup $H$ of index $p$
  which is isomorphic to $G$.

  Then the commensurator of $G$ is not finitely generated.
\end{prop}
\begin{proof}
  Suppose as above that $\Comm(G)$ is generated by a finite number of
  classes of isomorphisms $\alpha_i: H_i\to H_i'$, for $i=1,\dots,k$,
  where $H_i,H_i'$ are subgroups of finite index in $G$.  Set
  $n=\max\{\subindex G{H_i},\subindex G{H'_i}:\,i=1,\dots,k\}$.

  Now take a prime number $p>n$. By assumption, there exists a
  subgroup $H$ of index $p$ in $G$ and an isomorphism $\beta:G\to H$.

  Clearly $[\beta]\in \Comm(G)$. By Lemma~\ref{lem:subindex}, the
  class $[\beta]$ can be realized by an isomorphism $\alpha:A\to B$,
  where $A,B$ are subgroups of finite index in $G$ and subindex at
  most $n$.  By Lemma~\ref{lem:restrict}, the automorphisms $\beta$
  and $\alpha$ coincide on $A$.

  By Lemma~\ref{lem:extend}, the automorphism $\beta$ can be extended
  to an isomorphism $\varphi:G\to BH$. Note that $BH=G$ because the
  indices of $B$ and $H$ are coprime. We have reached a contradiction.
\end{proof}

\begin{proof}[Proof of Corollary~\ref{cor:surf}]
  It is well known that $G$ has the unique root property (e.g.\
  because $G$ is a torsion-free hyperbolic group, see
  Lemma~\ref{lem:urp}; or more directly because $G$ is a group of
  diagonalizable $2\times 2$ matrices).

  First consider the case in which $G$ is a free group with basis
  $X=\{x,y,\dots\}$.  Given an integer $p>1$, let $G\rightarrow
  \Z/p\Z$ be the homomorphism which sends $x$ to $1$ and all other
  elements of $X$ to $0$. The kernel $H$ of this homomorphism is free
  on $Y=\{x^p,y,x^{-1}yx,\dots,x^{1-p}yx^{p-1},\dots\}$. Clearly, the
  automorphism of $H$ which exchanges $y$ and $x^p$ and fixes all
  other elements of $Y$ cannot be extended to an automorphism of $G$,
  because $x^p$ is primitive in $H$ but not in $G$. By
  Theorem~\ref{thm:infgen}, $\Comm(G)$ is not finitely generated.

  It is convenient to translate this argument to topological
  language. The group $G$ is the fundamental group of a rose $R$, with petals
  indexed by the elements of $X$. Consider the regular degree-$p$ cover $\widetilde R$
  of $R$, in which a petal (say $x$) has been unfolded $p$ times to a
  ``gynoecium'' (central circle) $\tilde x$. Consider another petal
  $y$ of $R$, and its lift $\tilde y$. The graph $\widetilde R$ is
  homotopy equivalent to a rose, so admits a homotopy equivalence
  $\varphi$ that exchanges $\tilde x$ and $\tilde y$ while fixing (up
  to homotopy) the other petals. Then $\varphi$ cannot be induced by a
  homotopy equivalence of $R$, because it fixes (up to homotopy) some
  lift of $y$ while moves another.

  Consider now the case in which $G=\pi_1(S)$ where $S$ is a compact
  closed surface of negative Euler characteristic.  By Lemma
  \ref{lem:comm} we may assume that $S$ is orientable.  Given an
  integer $p>1$, let $\widetilde{S}\rightarrow S$ a regular degree-$p$
  cover of $S$. Clearly $\widetilde{S}$ is of strictly more negative
  Euler characteristic.

  Consider two handles $x,x'$ of $\widetilde{S}$ covering the same
  handle of $S$, and a handle $y$ that covers a different handle of
  $S$. Let $T$ be a neighbourhood of $x$, $y$ and a path connecting
  $x$ to $y$ that is homeomorphic to a punctured $2$-handlebody.  Let
  $\varphi$ be the homeomorphism of $\widetilde S$ that exchanges $x$
  and $y$ and is homotopic to the identity outside of $T$. Again,
  $\varphi$ is not induced by a homeomorphism of $S$, since it moves
  $x$ while it fixes its conjugate $x'$ . Therefore, the automorphism
  induced by $\varphi$ on $\pi_1(\widetilde{S})$ cannot be extended to
  an automorphism of $\pi_1(S)$. As above, Theorem~\ref{thm:infgen}
  completes the proof.
\end{proof}

\section{Free products of groups}
We prove in this section that many free products have infinitely
generated commensurator.

\begin{lem}\label{lem:extendAB}
  Let $H$ be a finite-index subgroup of $G$; assume $G$ is generated
  by the union of two subgroups $A,B$ and has the unique root property; let $\varphi:H\to H$
  be an automorphism. If $\varphi\neq id$, but $\varphi\restrict{H\cap
    A}=\id$,  $\varphi\restrict{H\cap B}=\id$, then $\varphi$ does not extend to
  an automorphism of~$G$.
\end{lem}
\begin{proof}
  Write $n=\index GH$, and let $\psi:G\to G$ be an extension of
  $\varphi$.  Take an arbitrary element $a\in A$. Then $a^{n!}\in
  H\cap A$, and so $\psi(a^{n!})=a^{n!}$. Since $G$ has the unique
  root property, we get $\psi(a)=a$, that is $\psi$ is the identity on
  $A$. Analogously $\psi$ is the identity on $B$, and hence $\psi=\id$,
  a contradiction.
\end{proof}

\begin{thm}\label{lem:free}
  Suppose that two nontrivial groups $A$ and $B$ have the unique root
  property, and at least one of them has finite quotients of
  arbitrarily large prime order. Then $\Comm(A\ast B)$ is not finitely
  generated.
\end{thm}
\begin{proof}
  Write $G=A\ast B$, and assume without loss of generality that $A$
  has arbitrarily large quotients.  Consider a normal subgroup
  $H\triangleleft G$ of finite index $n>1$ and containing $B$, e.g.\
  the kernel of the map $A\ast B\to Q\ast 1$ for a finite quotient $Q$
  of $A$.  By Kurosh's theorem, there exists a nontrivial splitting of
  the form $H=(H\cap A)\ast (H\cap B)\ast C$ with $C\neq 1$.  Let $b$
  be a nontrivial element of $H\cap B$; there is some, because $H\cap
  B=B$ is nontrivial.  Consider the automorphism $\varphi$ of $H$,
  which is the identity on $H\cap A$ and on $H\cap B$ and is
  conjugation by $b$ on $C$.

  By Lemma~\ref{lem:extendAB}, this $\varphi$ does not extend to $G$. We
  conclude by Theorem~\ref{thm:infgen}.
\end{proof}

This gives another proof of Corollary~\ref{cor:surf} for free groups
of rank $n\geq 2$: if $G=F_n$, take $A=\Z$ and $B=F_{n-1}$ and apply
Theorem~\ref{lem:free}. Another proof of Corollary~\ref{cor:surf} for surface groups
follows from Theorem~\ref{thm:HNN} or \ref{thm:freeprod}.

Note that the abstract commensurator of a free group admits an elegant
description through automata, see~\cite{MNS}. Lemma~\ref{lem:subindex}
essentially says that, given a finite collection of elements in the
commensurator of $F_m$, there exists a finite alphabet (with $n$
letters in the lemma's notation) such that these elements are
represented by automata on that alphabet.

\section{Groups splitting over $\Z$}
Following on Theorem~\ref{lem:free}, we now apply
Theorem~\ref{thm:infgen} to free products with amalgamation and HNN
extensions. In the proof
%of Theorem~\ref{thm:freeprod}
we will use certain automorphisms of $G$, called \emph{Dehn twists}.

\begin{lem}\label{lem:BS}
  Let $k$ be an integer, and consider the Baumslag-Solitar group
  \[G=\langle a,t\mid tat^{-1}=a^k\rangle.\]
  Then $\Comm(G)$ is not finitely generated.
\end{lem}
\begin{proof}
  If $k=0$, then $G$ is infinite cyclic and the statement obviously
  holds; so assume $k\neq 0$. Let $p$ be a prime $>k$. Consider the
  endomorphism $\psi: G\to G$ sending $t$ to $t$ and $a$ to $a^p$. We
  prove that $\psi$ is injective, and that $\psi(G)$ has index $p$ in
  $G$; the conclusion then follows from Proposition~\ref{prop:infgen}.

  We have $G=\Z[1/k]\rtimes\langle t\rangle$, and $\psi$ is given by
  $\psi(x,t^i)=(px,t^i)$; so $\psi$ is an injective endomorphism. Its
  image is $p\Z[1/k]\rtimes\langle t\rangle$, which has index $p$
  because $p$ and $k$ are coprime.
\end{proof}

\begin{thm}\label{thm:HNN}
  Let $G=A\ast_C$, where $C$ is an infinite cyclic group. If $G$ has the
  unique root property, then $\Comm(G)$ is not finitely generated.
\end{thm}
\begin{proof}
  The group $G$ has the presentation $\langle A,t\,|\,
  t^{-1}Ct=D\rangle$, where $t$ is stable letter and $C=\langle
  c\rangle$, $D=\langle d\rangle$ are associated subgroups of $A$.

  Consider $n\geq 3$ and let $H_n$ be the kernel of the
  homomorphism $G\to\Z/n\Z$ sending $A$ to $0$ and $t$ to $1$. Then
  $H_n$ is also an HNN extension, which has the following
  presentation:
  \begin{align*}
    H_n &= \bigl\langle
    K,s\,\big|\,s^{-1}(t^{n-1}Ct^{1-n})s=D\bigr\rangle,\text{ where }\\
    K &= A\underset{\scriptscriptstyle C=tDt^{-1}}{\ast}
    tAt^{-1}\underset{\scriptscriptstyle tCt^{-1}=t^2Dt^{-2}}{\ast} t^2At^{-2}\ast\dots\underset{\scriptscriptstyle t^{n-2}Ct^{2-n}=t^{n-1}Dt^{1-n}}{\ast}t^{n-1}At^{1-n}
  \end{align*}
  and the stable letter $s$ corresponds to $t^n$ in $G$. Consider the
  automorphism $\varphi$ of $H_n$ which fixes the base $K$ of the HNN
  extension and sends $s$ to $sd$.  Suppose that $\varphi$ can be
  extended to an automorphism $\psi$ of $G$. Then, since
  $tAt^{-1}\leqslant K$, for any $a\in A$, we have
  $tat^{-1}=\varphi(tat^{-1})=\psi(tat^{-1})=\psi(t)\psi(a)\psi(t^{-1})=\psi(t)a\psi(t)^{-1}$,
  and so $t^{-1}\psi(t)\in C_G(A)$.

  Now either the HNN extension is ascending ($C=A$), in which case $G$
  is a Baumslag-Solitar group, and we are done by Lemma~\ref{lem:BS};
  or $C_G(A)=Z(A)$, and we get $\psi(t)=ta$ for some $a\in
  Z(A)\setminus \{1\}$.  We then have
  $t^nd=sd=\varphi(s)=\psi(t^n)=(ta)^n$; hence
  \begin{equation}\label{eq:1}
    \underbrace{t^{-1}(t^{-1}(\dots (t^{-1}(t^{-1}}_{n-1}(a)ta)ta)\dots )ta)tad^{-1}=1.\tag{$\dagger$}
  \end{equation}
Another cyclic form of this equation is
\begin{equation}\label{eq:2}
  tata\dots tat(ad^{-1})\underbrace{t^{-1}t^{-1}\dots
    t^{-1}t^{-1}}_{n-1}a=1.\tag{$\ddagger$}
\end{equation}
Using normal form in HNN extensions we deduce from~\eqref{eq:1} that
$a\in C$, and from~\eqref{eq:2} that $ad^{-1}\in D$.  Thus,
$a=c^p=d^q$ for some nonzero $p,q$. Since $a\in Z(A)$ and $Z(A)$ is
closed under taking roots (since $G$ has unique root property), we get
$c,d\in Z(A)$. In particular, $\langle c,d\rangle$ is a torsion-free
Abelian group satisfying $c^p=d^q$. Therefore this group is cyclic,
that is $c=z^l$ and $d=z^r$ for some $z\in Z(A)$ and $l,r\in \Z$.
Thus, we have
\begin{equation}\label{eq:3}
  a=z^{pl}\qquad\text{and}\qquad t^{-1}z^lt=z^r.\tag{$\mathchar "278$}
\end{equation}

We now analyze Equation~\eqref{eq:1} deeper. Using~\eqref{eq:3}, we
successively deduce
$$\begin{array}{l}
a=z^{pl},\\ t^{-1}(a)ta=z^{pl(1+(r/l))},\\
t^{-1}(t^{-1}(a)ta)ta=z^{pl(1+(r/l)+(r/l)^2)},\\
\vdots \\ \underbrace{t^{-1}(\dots (t^{-1}(t^{-1}}_{n-2}(a)ta)ta)\dots )ta=z^{pl(1+(r/l)+\dots +(r/l)^{n-2})},\\
\end{array}
$$

Finally, we obtain from~\eqref{eq:1} that

$$1=t^{-1}(t^{-1}(\dots (t^{-1}(t^{-1}(a)ta)ta)\dots )ta)tad^{-1}=z^{pl(1+(r/l)+\dots +(r/l)^{n-1})-r},$$

so

$$pl(1+(r/l)+\dots +(r/l)^{n-1})=r.$$

Equivalently,
$$p(l^{n-1}+rl^{n-2}+\dots +r^{n-1})=rl^{n-1}.$$

Note that $\gcd(r,l)=1$, otherwise, using the unique root property of
$G$, we could extract a root from $tz^lt^{-1}=z^r$ and get a wrong
equation. Hence $(l^{n-1}+rl^{n-2}+\dots +r^{n-1})$ has no nontrivial
common divisor neither with $r$, nor with $l$. Therefore
$(l^{n-1}+rl^{n-2}+\dots +r^{n-1})=\pm 1$.  Since $n\ge3$, this is
possible only if $l=1,r=-1$ or $l=-1,r=1$.  In that last case, $G$ has
the presentation $G=\langle A,t\,|\, t^{-1}zt=z^{-1}\rangle$. Then its
index $2$ subgroup $H_2$ has the presentation

$$H_2=\Bigl\langle \Bigl( A\underset{z=tz^{-1}t^{-1}}{\ast} tAt^{-1}\Bigr), s\,|\,
s^{-1}zs=z\Bigr\rangle,$$ where $s$ corresponds to $t^2$ in $G$. Thus,
if we replace $G$ by $H_2$ we will have $l=r=1$.  Thus, after possible
replacement, $\varphi$ cannot be extended to an automorphism of $G$
and we conclude by Theorem~\ref{thm:infgen}.
\end{proof}

\begin{lem}\label{prop:aut}
  Consider $G=G_1\ast_C G_2$, where $C$ is infinite cyclic. If $G_2$
  is Abelian, assume furthermore that $G_2=K\oplus L$ with $C\leqslant
  K$ and $|L|>2$.

  Then $G$ has a nontrivial automorphism $\varphi$ which fixes $G_1$.
\end{lem}
\begin{proof}
  It is enough to define a nontrivial automorphism
  $\psi:G_2\rightarrow G_2$, such that $\psi\restrict C=\id$.  Then such
  $\psi$ can be obviously extended to the desired $\varphi$.

  If $C$ does not lie in $Z(G_2)$, we define $\psi$ as conjugation
  by a generator of $C$.  If $C$ lies in $Z(G_2)$ and $G_2$ is not
  Abelian, we take an element $g\in G_2\setminus Z(G_2)$ and define
  $\psi$ as conjugation by $g$. Consider finally $G_2$ Abelian,
  with $G_2=K\oplus L$. If $2L\neq0$, define $\psi:G_2\to G_2$ by
  $\psi(x,y)=(x,-y)$ for $x\in K,y\in L$; while if $2L=0$ then $L$ is
  an $\mathbb F_2$-vector space of dimension $>1$, so admits a
  non-trivial automorphism $\psi'$. Set then $\psi(x,y)=(x,\psi'(y))$.
\end{proof}

\begin{thm}\label{thm:freeprod}
  Let $G$ be $A\ast_C B$,  where $C$ is an infinite cyclic
  subgroup distinct from $A$ and $B$. Suppose that
  \begin{enumerate}
  \item $G$ has the unique root property;
  \item $A/C^A$ maps homomorphically onto $\Z$;
  \item if $B$ is Abelian, then $B$ maps homomorphically onto $\Z$.
  \end{enumerate}
  Then $\Comm(G)$ is not finitely generated.
\end{thm}
Note that (2) is satisfied as soon as $G$ maps onto $\Z\times\Z$,
and~(3) is satisfied as soon as $B$ is finitely generated.

\begin{proof}
  We first show that we may assume additionally that the following
  condition is satisfied:
  \[\tag{4} |B:C| \text{ is infinite.}\]

  Suppose that the index $|B:C|$ is finite, so $B$ is virtually
  cyclic. Since $G$ is torsion-free, $B$ is infinite cyclic.  Let
  $1,b,b^2,\dots ,b^{n-1}$ be representatives of $B$ modulo $C$. Note
  that $n\geq 2$, since $B\neq C$.  Let $\varphi:A\ast_C
  B\rightarrow \Z/n\Z$ send $A$ onto $\Z/n\Z$ and $B$ to $0$.  The
  kernel $G_1$ of $\varphi$ can be presented as the free
  product of groups $b^{-i}Ab^i$ for $i\in\{0,1,\dots, n-1\}$,
  amalgamated over the common subgroup $C$. Therefore
  $G_1=A\ast_CB_1$, where $B_1$ is the free product of $A^{b^i}$ for
  $i\in\{1,\dots, n-1\}$, amalgamated over $C$.  Then $B_1=A^b\ast_C
  V=A^b\ast_{C^b} V$ for some group $V$. It follows $B_1/V^{B_1}\cong
  A/C^A$ and so $B_1$ satisfies Condition~(3). Moreover, $B_1$
  satisfies Condition~(4), since $B_1$ contains $A^b$ and
  $|A^b:C|=|A:C|$ is infinite.  Since $G_1$ has finite index in $G$,
  we have $\Comm(G)\cong \Comm(G_1)$.  Therefore, replacing $G$ by
  $G_1$ if necessary, we may assume that Conditions~(1--4) are
  satisfied.

  We then show that we may assume additionally that the following
  condition, which in required in Lemma~\ref{prop:aut}, is satisfied:

  \[\tag{5}\text{if $B$ is Abelian, then $B=K\oplus L$, 
    with $C\leqslant K$ and $|L|>2$.}\]

  Suppose that $B$ is Abelian. By Condition~(3), there is an
  epimorphism $\psi: B\rightarrow\Z$. Thus, $B=\ker\psi\oplus\Z$. If
  $C\leqslant \ker\psi$, we are done. If $C\not\leqslant\ker\psi$,
  then $\psi(C)$ has finite index in $\psi(B)$. Since $C$ is infinite
  cyclic, we have $\ker\psi \cap C=\{0\}$.  Denote $B_1=\langle
  \ker\psi, C\rangle$. Then $B_1=\ker\psi\oplus C$ and the index
  $n=|B:B_1|$ is finite. Hence $B_1$ satisfies Conditions~(3--4).  In
  particular, $\ker\psi$ is infinite. Therefore, $B_1$ satisfies
  Condition~(5).

  If $n=1$, then $B=B_1$, and so $B$ satisfies Conditions (3--5).
  Suppose then $n\geq 2$ and let $T=\{b_1,\dots,b_n\}$ be a
  transversal of $B_1$ in $B$.  Consider $H=\langle
  A,B_1\rangle^G$. Then $H$ has index $n$ in $G$; hence $\Comm
  (G)\cong \Comm (H)$.  Moreover, $T$ is a transversal of $H$ in $G$.
  Consider the induced decomposition of $H$ as the fundamental group
  of a graph of groups (see~\cite{Se}): it has the shape of a star;
  there is a central vertex with vertex group $B_1$ and $n$ outer
  vertices with vertex groups $A^b$ for $b\in T$. All edge groups are
  $C$. We can rewrite this decomposition in the form $H= B_1\ast_C
  A_1$, where $A_1=A\ast_C V $ for some $V$. The group $A_1/C^{A_1}$
  can be mapped homomorphically onto $A_1/V^{A_1}=A/C^A$ and so
  Condition~(2) is satisfied by $A_1$.
  
  In summary, without loss of generality we assume that
  Conditions~(1--5) are satisfied for the original $G$.

  We now show that for any prime number $p>1$, there exists a normal
  subgroup $H$ of index $p$ in $G$, and an automorphism of $H$ that
  does not extend to an automorphism of $G$.  Then
  Theorem~\ref{thm:infgen} will complete this proof.

  By~(2), the quotient group $A/C^A$ can be homomorphically mapped
  onto $\Z$ and further onto $\Z/p\Z$. Let $N\triangleleft A$ be the
  kernel of the composition of these epimorphisms, and set $H=\langle
  N,B\rangle^G$. Then $C\leqslant H\triangleleft G$ and $\index
  GH=p$. Consider the induced decomposition of $H$ as the fundamental
  group of a graph of groups: it has the shape of a star; there is a
  central vertex with the vertex group $N$ and $p$ outer vertices with
  the vertex groups $B^a$ for $a$ in a transversal of $N$ in $A$.

  In particular, $H=U\ast_{C^a} B^a$ for some $a\notin A$ and some
  subgroup $U$ containing $B$ and $N$. By Lemma~\ref{prop:aut}, there
  is a non-trivial automorphism $\varphi$ of $H$ fixing $U$. We
  conclude by Lemma~\ref{lem:extendAB} that $\varphi$ cannot be
  extended to an automorphism of~$G$.
\end{proof}

\noindent To prove Corollary~\ref{cor:2}, we recall a theorem by Paulin:
\begin{thm}[\cite{Pau}]\label{thm:paulin}
  Suppose $G$ is a word-hyperbolic group with infinite $\Out(G)$. Then
  $G$ splits over a virtually cyclic group.
\end{thm}

\begin{proof}[Proof of Corollary~\ref{cor:2}]
  By Theorem~\ref{thm:paulin}, $G$ splits over a virtually cyclic
  subgroup, that is $G=A\ast_C B$ or $G=A\ast_C$, where $C$ is
  virtually cyclic. Since $G$ is torsion-free, $C=1$ or $C=\Z$. If
  $C=1$, we apply Theorem~\ref{lem:free}. If $G$ is an HNN extension,
  we apply Theorem~\ref{thm:HNN}.

  If $C=\Z$ then, since $G$ maps onto $\Z^2$, its quotient $G/C^G$
  maps onto $\Z$. Since $G/C^G=A/C^A\ast B/C^B$, one of the groups
  $A/C^A$ or $B/C^B$ maps onto $\Z$. If $A$ or $B$ is Abelian, it is
  cyclic, since $G$ is a torsion-free hyperbolic group. We conclude by
  Theorem~\ref{thm:freeprod}.
\end{proof}

\section{Examples}
We conclude in this section with a few examples showing that
additional conditions are required on a hyperbolic group or on a free
product with amalgamation to ensure that its commensurator is
infinitely generated.

The following construction was generously indicated to us by
Marc Lackenby. Consider a complicated-enough knot $K\subset S^2\times
S^1$; namely, the mapping torus of a complicated-enough braid $\tilde
K\subset S^2\times[0,1]$. Let $\mu$ be a small loop in $S^2\times
S^1\setminus K$ around $K$.

Set $\Delta=\pi_1(S^2\times S^1\setminus K)$. Then, for $n$ large
enough, $\Gamma:=\Delta/\langle\mu^n\rangle^\Delta$ is
hyperbolic~\cite{Boi}, and is a non-arithmetic lattice in
$G:=\operatorname{PSL}_2(\mathbb C)$. By rigidity (see the
Introduction), $\Comm(\Gamma)=\Comm_G(\Gamma)$; and
by~\cite[Theorem~IX.1.B]{Ma}, $\Gamma$ has finite index in
$\Comm(\Gamma)$, so in particular $\Comm(\Gamma)$ is finitely
generated. We deduce:
\begin{thm}[=Theorem~\ref{cor:Z}(1)]
  There exist word-hyperbolic groups with finitely generated
  commensurator.
\end{thm}
(Note of course that $\Gamma$ is not torsion-free).

Recall that a group $G$ is called {\it complete} if it has trivial
center and no outer automorphisms.  A group is called {\it perfect} if
it equals its own commutator subgroup. A subgroup $C$\ of a group $G$
is called {\it malnormal} if $C\cap g^{-1}Cg=1$ for every $g\in
G\setminus C$.  We will use the following result of V.N.~Obraztsov
(see Corollary~3 in~\cite{Obr} and its proof).

\begin{thm}[\cite{Obr}]\label{thm:obr} There exists a 2-generated
  simple complete torsion-free group $G$ in which every proper
  subgroup is infinite cyclic.
\end{thm}

We note that such a group $G$ has maximal cyclic subgroups; indeed
otherwise it would contain an infinite ascending sequence of cyclic
subgroups; its union cannot be cyclic, and so it must coincide with
$G$. This is impossible since $G$ is finitely generated.

\begin{lem}\label{thm:urp} Let $G$ be a group as in
  Theorem~\ref{thm:obr}. Then every maximal cyclic subgroup of $G$ is
  malnormal. Moreover, $G$ has the unique root property.
\end{lem}

\begin{proof} Let $\langle z\rangle$ be a maximal cyclic subgroup in
  $G$ and suppose that it is not malnormal, that is $\langle
  z\rangle\cap g^{-1}\langle z\rangle g\neq 1$ for some $g\in
  G\setminus \langle z\rangle$.  Then $z^s=g^{-1}z^tg$ for some
  nonzero $s,t$.  Moreover, the subgroup $\langle g,z\rangle$ is
  larger than $\langle z\rangle$, so it is noncyclic and therefore
  equals $G$.

  If $g^{-1}zg\notin \langle z\rangle$, then $\langle g^{-1}zg,z
  \rangle=G$ and hence $z^s$ lies in the center of $G$, a
  contradiction.

  If $g^{-1}zg\in \langle z\rangle$, then $g^{-1}zg=z^k$ for some
  $k$. If $|k|\geq 2$, then $\langle z\rangle$ is not maximal, a
  contradiction.  If $|k|=1$, then $g^2$ lies in the center of
  $G=\langle g,z\rangle$, again a contradiction.

  Now we prove that $G$ has the unique root property. Suppose that for
  some $x,y\in G$ holds $x^n=y^n$, $n\neq 0$.  If $x,y$ generate a
  cyclic group, then clearly $x=y$.  If they generate a noncyclic
  group, then $\langle x,y\rangle=G$. But then $x^n$ lies in the
  center of $G$, so $x^n=1$, and so $x=1$. Similarly
  $y=1$.
\end{proof}

\begin{thm}[=Theorem~\ref{cor:Z}(2)]\label{thm:example} There exists a
  3-generated group $G=G_1\underset{{u_1=u_2}}{\ast} G_2$ such that
  \begin{enumerate}
  \item $G$ is torsion-free;\label{cor:Z:1}
  \item $G/[G,G]=\Z$ and $u_i\notin [G,G]$;\label{cor:Z:2}
  \item $G$ has the unique root property;\label{cor:Z:3}
  \item $\Comm(G)=\Aut(G)$;\label{cor:Z:4}
  \item $\Aut(G)$ is generated by inner automorphisms, a Dehn twist
    along $\langle u_i\rangle$ and possibly one extra automorphism
    which interchanges $G_1$ and $G_2$. In particular, $\Aut(G)$ is
    finitely generated.\label{cor:Z:5}
  \end{enumerate}
\end{thm}
\begin{proof} Let $H_1$, $H_2$ be two groups as in
  Theorem~\ref{thm:obr}. In each $H_i$ we choose an element $h_i$,
  generating a maximal cyclic subgroup. We set $G_i=H_i\times A_i$,
  where $A_i=\langle a_i\rangle$ is an infinite cyclic group, take
  $u_i=h_ia_i$ and define $G=G_1\underset{{u_1=u_2}}{\ast} G_2$.

  We denote by $u$ the image of $u_i$ in $G$. Note that the
  centralizer of the subgroup $\langle u\rangle$ in $G$ has the
  following structure: $C_G(u)=\langle u\rangle \times Z$, where
  $Z=\langle A_1,A_2\rangle$. Since $A_i\cap \langle u_i \rangle =1$,
  we have $Z=A_1\ast A_2\cong F_2$.

  {\it Remark.} Using Lemma~\ref{thm:urp} one can prove the following
  important property: if for some $g\in G$ we have that
  $g^{-1}u^sg=u^t$ for some nonzero $s,t$, then $s=t$ and $g\in
  C_G(u)$.

  We are now ready to prove the statements. \eqref{cor:Z:1} is weaker
  than \eqref{cor:Z:3}.

  \medskip

  \eqref{cor:Z:2} This statement follows from the fact that $H_1,H_2$
  are perfect.

  \medskip

  \eqref{cor:Z:3} Assume the converse: there are two different
  elements $x,y\in G$ such that $x^n=y^n$.  We will analyze the action
  of $x$ and $y$ on the Bass-Serre tree $T$ associated with the
  decomposition $G=G_1\underset{{u_1=u_2}}{\ast} G_2$. Clearly, $x,y$
  are either both elliptic or both hyperbolic.  For any edge $e$ of
  $T$ let $\alpha(e)$ and $\omega(e)$ denote the initial and the
  terminal vertices of $e$ respectively.

  \medskip

  {\it Case 1.} Suppose that $x,y$ are both elliptic. If they
  stabilize the same vertex of $T$, then (after conjugation) we may
  assume that $x,y\in G_i$ for some $i=1,2$. Then, using
  Lemma~\ref{thm:urp}, we conclude $x=y$.

  Suppose that $x$ and $y$ do not stabilize the same vertices of $T$.
  We choose the shortest path $p=e_1e_2\dots e_m$ in $T$ such that
  $x\in \Stab(\alpha(e_1))$ and $y\in \Stab(\omega(e_m))$. Then this
  path is stabilized by $x^n(=y^n)$, in particular, $e_1$ is
  stabilized by $x^n.$ By conjugating and renaming the factors, we can
  assume that $\Stab(\alpha(e_1))=G_1$, $\Stab(\omega(e_1))=G_2$ and
  $\Stab(e_1)=G_1\cap G_2=\langle u\rangle$. Since $x\in G_1$, we have
  $x=za_1^k$ for some $z\in H_1$, $k\in \Z$. And since $x^n\in G_1\cap
  G_2$, we have $x^n=z^na_1^{kn}=u^{kn}=h_1^{kn}a_1^{kn}$. In particular,
  $z^n=h_1^{kn}$ and so $z=h_1^k$ by Lemma~\ref{thm:urp}. This implies
  that $x=h_1^ka_1^k=u_1^k\in G_1\cap G_2=\Stab(e_1)$, a contradiction
  to the minimality of the path $p$.

  \medskip

  {\it Case 2.} Suppose that $x,y$ are both hyperbolic. Since
  $x^n=y^n$, the axes of $x$ and $y$ coincide and $x^{-1}y$ and
  $x^{-2}y^2$ stabilize this axis. By conjugating we may assume that
  $x^{-1}y$ and $x^{-2}y^2$ lie in $G_1\cap G_2$.  Thus $y=xu^k$ for
  some $k\in \Z$ and so $y^2=x^2\cdot x^{-1}u^kxu^k$. Hence
  $x^{-1}u^kx\in G_1\cap G_2$.  By the remark at the beginning of this
  proof, we conclude that $x\in C_G(u)$. Similarly, $y\in
  C_G(u)$. Since $C_G(u)=\langle u\rangle \times Z\cong \langle
  u\rangle \times F_2$ has the unique root property, we conclude from
  $x^n=y^n$ that $x=y$.

  \medskip

  (\ref{cor:Z:4},\ref{cor:Z:5}) First we describe finite index
  subgroups of $G$. Let $B$ be a subgroup of finite index $m$ in $G$,
  and let $N$ be a normal subgroup of finite index in $G$ such that
  $N\leqslant B$. Since $H_i$ does not contain proper finite index
  subgroups, we have $G_i\cap N=(H_i\times \langle a_i\rangle)\cap
  N=H_i\times \langle a_i^{m_i}\rangle$ for some $m_i\in \Z$. Then $N$
  contains the normal closure of $\langle H_1,H_2\rangle$ in $G$. The
  factor group of $G$ by this normal closure is isomorphic to
  $\Z$. Therefore $B$ is normal and coincides with the preimage of
  $m\Z$.

  We claim that $B=(H_1\times \langle
  a_1^m\rangle)\underset{{u_1^m=u_2^m}}{\ast}(H_2\times \langle
  a_2^m\rangle)$.  Simplifying notations we write $G_{i,m}=H_i\times
  \langle a_i^m\rangle$ and
  $G(m)=G_{1,m}\underset{{u_1^m=u_2^m}}{\ast}G_{2,m}$. Thus we want to
  prove that $B=G(m)$.

  It is enough to prove that $G(m)$ is normal in $G$ (then clearly
  $G/G(m)\cong \Z/m\Z$ and so $B=G(m)$).  Note that $G(m)=\langle
  a_1^m,a_2^m,H_1,H_2 \rangle$ and $G=\langle a_1,a_2,H_1,H_2
  \rangle$.  Preparing to conjugate, we deduce from the equations
  $h_1a_1=h_2a_2$ and $[h_i,a_i]=1$ the following:
  $$a_1a_2^{-1}=h_1^{-1}h_2\in H_1H_2\leqslant G(m),$$
  $$a_1^{-1}a_2=h_1h_2^{-1}\in H_1H_2\leqslant G(m).$$
  Then for $\varepsilon\in \{-1,1\}$ we have
  $$\begin{array}{ll}
    a_1^{\varepsilon}a_2^ma_1^{-\varepsilon} & =(a_1^{\varepsilon}a_2^{-\varepsilon})a_2^m(a_1^{\varepsilon}a_2^{-\varepsilon})^{-1}\in G(m),\vspace*{2mm}\\
    a_1^{\varepsilon}H_2a_1^{-\varepsilon} & =(a_1^{\varepsilon}a_2^{-\varepsilon})a_2^{\varepsilon}H_2a_2^{-\varepsilon}(a_1^{\varepsilon}a_2^{-\varepsilon})^{-1}=
    (a_1^{\varepsilon}a_2^{-\varepsilon})H_2(a_1^{\varepsilon}a_2^{-\varepsilon})^{-1}\leqslant G(m).
  \end{array}$$
  By symmetry we get $a_2^{\varepsilon}a_1^ma_2^{-\varepsilon}\in
  G(m)$ and $a_2^{\varepsilon}H_1a_2^{-\varepsilon}\leqslant
  G(m)$. This completes the proof that $G(m)$ is normal in $G$ and so
  $B=G(m)$. Thus, for every natural $m$ there is a unique subgroup of
  index $m$ in $G$; it has the
  form
  \begin{equation}\label{eq:Gm}
    G(m)=G_{1,m}\underset{u_1^m=u_2^m}{\ast}G_{2,m}.\tag{$\#$}
  \end{equation}

  We now investigate which isomorphisms can appear in $\Comm(G)$.  Let
  $n,m$ be two natural numbers and let $\alpha: G(n)\rightarrow G(m)$
  be an isomorphism.  We claim that $G_{i,n}$ is nonsplittable over a
  cyclic subgroup.  Indeed, suppose $G_{i,n}=K\ast_L M$, where $L$ is
  a cyclic group. If one of the indices $\index KL$ or $\index ML$ is
  larger than 2, then $G_{i,n}$ and hence its direct factor $H_i$
  would contain a noncyclic free group, contradicting the properties
  of $H_i$.  If $\index KL=\index ML=2$, then $G_{i,n}\cong \Z/2\Z\ast \Z/2\Z$
  or $G_{i,n}\cong \Z\ast_{2\Z=2\Z}\Z$, again absurd in regard of
  Theorem~\ref{thm:obr}. An analogous reasoning shows that $G_{i,n}$
  cannot be a nontrivial HNN extension over a cyclic group.

  This implies that $\alpha(G_{i,n})$ is also nonsplittable over a
  cyclic subgroup and so is conjugate to $G_{1,m}$ or to $G_{2,m}$.

  \medskip

%It remains to prove that $\Aut(G)$ is finitely generated.

  {\it Case 1.} Suppose that $\alpha(G_{1,n})$ is conjugate to
  $G_{1,m}$ and $\alpha(G_{2,n})$ is conjugate to $G_{2,m}$.
  Composing $\alpha$ with an appropriate conjugation, we can assume
  that $\alpha(G_{1,n})\leqslant G_{1,m}$ and
  $\alpha(G_{2,n})\leqslant gG_{2,m}g^{-1}$ for some $g\in G(m)$. We
  prove that $\alpha(G_{2,n})\leqslant G_{2,m}$.  We can assume that
  $g$, written in reduced form with respect to the amalgamated
  product~\eqref{eq:Gm}, is either empty or starts with an element of
  $G_{2,m}\setminus \langle u^m\rangle$ and ends with an element of
  $G_{1,m}\setminus \langle u^m\rangle$.

  Suppose that $g$ is nonempty and write it in reduced form:
  $g=g_1g_2\dots g_{2k-1}g_{2k}$, where $g_i\in G_{1,m}\setminus
  \langle u^m\rangle$ if $i$ is even and $g_i\in G_{2,m}\setminus
  \langle u^m\rangle$ if $i$ is odd.  The element $\alpha(u^n)$ lies
  in $\alpha(G_{1,n})\cap \alpha(G_{2,n})=G_{1,m}\cap gG_{2,m}g^{-1}$,
  hence it can be written as $\alpha(u^n)=g_1g_2\dots
  g_{2k-1}g_{2k}vg_{2k}^{-1}g_{2k-1}^{-1}\dots g_2^{-1}g_1^{-1}$ for
  some $v\in G_{2,m}$ and the reduced form of this product consists of
  only one factor which lies in $G_{1,m}$. Therefore $v\in \langle
  u^m\rangle$ and $g_i\in C_{G_{2,m}}(u^m)\setminus \langle
  u^m\rangle$ for odd $i$ and $g_i\in C_{G_{1,m}}(u^m)\setminus
  \langle u^m\rangle$ for even $i$. This implies

  (a) $gu^mg^{-1}=u^m$;

  (b) $\alpha(G_{1,m})\cap \alpha(G_{2,m})=\langle u^m\rangle$;

  (c) if $w\in \langle u^m\rangle$, then the reduced form of
  $gwg^{-1}$ with respect to~\eqref{eq:Gm} is $w$;

  (d) if $w\in G_{2,m}\setminus \langle u^m\rangle$, then the reduced
  form of $gwg^{-1}$ is $g_1g_2\dots
  g_{2k-1}g_{2k}wg_{2k}^{-1}g_{2k-1}^{-1}\dots g_2^{-1}g_1^{-1}$; it
  starts and ends with elements from $G_{2,m}\setminus \langle
  u^m\rangle$ and contains at least one element from $G_{1,m}\setminus
  \langle u^m\rangle$.

  \medskip

  Using this we prove that the group generated by $G_{1,m}$ and
  $gG_{2,m}g^{-1}$ does not contain elements of $G_{2,m}\setminus
  \langle u^m\rangle$, and that will contradict the surjectivity of
  $\alpha$. Let $z$ be an arbitrary element of $\langle
  \alpha(G_{1,n}),\alpha(G_{2,n})\rangle$. We write $z$ as
  $z=z_1z_2\dots z_l$, so that $z_i$ lie alternately in
  $\alpha(G_{1,n})$ or in $\alpha(G_{2,n})$ and $l$ is minimal. First
  suppose that $l>1$. Then $z_i\notin \langle u^m\rangle$, otherwise
  one can unify two consecutive factors of $z_1z_2\dots z_l$ and
  decrease $l$. Therefore the following hold:

  (i) If $z_i\in \alpha(G_{1,n})$, then $z_i\in G_{1,n}\setminus
  \langle u^m\rangle$.

  (ii) If $z_i\in \alpha(G_{2,n})$, then $z_i\in g(G_{2,n}\setminus
  \langle u^m\rangle)g^{-1}$ by (a).  By (c--d) the reduced form of
  $z_i$ with respect to~\eqref{eq:Gm} starts and ends with elements
  from $G_{2,m}\setminus \langle u^m\rangle$ and contains at least one
  element from $G_{1,m}\setminus \langle u^m\rangle$.

  Therefore the normal form of $z$ is the product of normal forms of
  $z_i$'s, and so $z\notin G_{2,m}\setminus \langle u^m\rangle$.

  If $l=1$, then either $z\in \langle u^m\rangle$, or as above
  $z\notin G_{2,m}\setminus \langle u^m\rangle$. In both cases
  $z\notin G_{2,m}\setminus \langle u^m\rangle$.

  \medskip

  We have reached a contradiction. Thus $g$ is empty and so
  $\alpha(G_{i,n})\leqslant G_{i,m}$ for $i=1,2$.

  \bigskip

  {\it Case 2.} Suppose that $\alpha(G_{1,n})$ is conjugate to
  $G_{1,m}$ and $\alpha(G_{2,n})$ is also conjugate to
  $G_{1,m}$. Composing $\alpha$ with an appropriate conjugation, we
  can assume that, say, $\alpha(G_{1,n})\leqslant G_{1,m}$ and
  $\alpha(G_{2,n})\leqslant gG_{1,m}g^{-1}$ for some $g\in G(m)$. Then
  arguing as in Case 1 we obtain a contradiction independently of
  whether $g$ is empty or not.

  \medskip

  All other possible cases can be considered similarly. Thus (after an
  appropriate conjugation), we may assume that
  $\alpha(G_{1,n})=G_{1,m}$ and $\alpha(G_{2,n})=G_{2,m}$ or
  $\alpha(G_{1,n})=G_{2,m}$ and $\alpha(G_{2,n})=G_{1,m}$.  In
  particular, $\alpha(u^n)=u^{\varepsilon m}$ for some $\varepsilon\in
  \{-1,1\}$. We consider the first case (the second case is similar).

  Since $H_i$ has no infinite cyclic quotients, we obtain
  $\alpha(H_i)=H_i$. Since $\alpha$ carries the center of $G_{i,n}$ to
  the center of $G_{i,m}$, we have $\alpha(a_i^n)=a_i^{\sigma m}$ for
  some $\sigma\in \{-1,1\}$.  Since $H_i$ is complete,
  $\alpha\restrict{H_i}$ is conjugation by an element $w_i\in
  H_i$. Therefore,
  $\alpha(u^n)=\alpha(h_i^na_i^n)=w_ih_i^nw_i^{-1}a_i^{\sigma m}$.  On
  the other hand $\alpha(u^n)=u^{\varepsilon m}=h_i^{\varepsilon
    m}a_i^{\varepsilon m}$. Thus, we have
  $w_ih_i^nw_i^{-1}=h_i^{\varepsilon m}$ and $\sigma=\varepsilon$.  By
  Lemma~\ref{thm:urp}, $w_i=h_i^{k_i}$ for some $k_i$ and so
  $n=\varepsilon m$, which implies $n=m$ and $\sigma=\varepsilon=1$
  since $m,n\in\mathbb N$.  Then $\alpha\restrict{G_{i,m}}$ is
  conjugation by $w_i$, which is the same as conjugation by
  $h_i^{k_i}a_i^{k_i}=u_i^{k_i}$. Thus, $\alpha$ is a product of two
  Dehn twists.

  All inner automorphisms and Dehn twists, and the (possible)
  permutation of factors of $G(n)$ can be lifted to the corresponding
  automorphisms of $G$. Thus properties~(3) and~(4) are proven.

  \medskip

  Finally we prove that $G$ is 3-generated. Recall that $h_i$
  generates a maximal cyclic subgroup in $H_i$. First we choose an
  element $y_i\in H_i\setminus \langle h_i\rangle$, $i=1,2$, and then
  take a generator $x_i$ of a maximal cyclic subgroup of $H_i$
  containing $y_i$. Clearly, $x_i\in H_i\setminus \langle h_i\rangle$
  and also $h_i\in H_i\setminus \langle x_i\rangle$.

  We claim that the subgroup $F=\langle x_1,x_2,u_1\rangle$ coincides
  with $G$. In the proof we will use the equations
  $h_1a_1=u_1=u_2=h_2a_2$.  We have
  $[x_i,u_i]=[x_i,h_ia_i]=[x_i,h_i]\in H_i$. By Lemma~\ref{thm:urp},
  the subgroup $\langle x_i\rangle$ is malnormal in $H_i$ and so
  $[x_i,h_i]\notin \langle x_i\rangle$.  Then, by
  Theorem~\ref{thm:obr}, $\langle x_i, [x_i,u_i]\rangle=H_i$. In
  particular, $H_i\leqslant F$. Then $A_i=\langle a_i\rangle=\langle
  h_i^{-1}u_i\rangle \leqslant F$ and hence $G=\langle
  H_1,H_2,A_1,A_2\rangle=F$.
\end{proof}

Note that $G$ from the proof of Theorem~\ref{thm:example} cannot be
generated by 2 elements.  Indeed, if $G$ were 2-generated, then its
homomorphic image $H_1\underset{{h_1=h_2}}{\ast} H_2$ would be also
2-generated. But this is impossible in view of~\cite[Corollary~1]{W},
which states that if $B$ is an amalgamated product of type
$\underset{i=1}{\overset{n}{*}_C}B_i$ where $C\neq 1$, $C\neq B_i$,
and $C$ is malnormal in $B$, then $\operatorname{rank}(B)\geq n+1$.

\section{Acknowledgments}

This work started in the ``profinite groups'' conference in
Oberwolfach, in 2008. Conversations with Martin Bridson, Marc
Lackenby, Volodymyr Nekrashevych, and Alexander Ol'shanskij are
gratefully acknowledged.  The second named author thanks the MPIM at
Bonn for its support and excellent working conditions during the fall
2008, while this research was conducted.

\end{document}